\documentclass[a4paper]{amsart}
\usepackage[english]{babel}
\usepackage{amssymb}
\usepackage{amsmath}
\usepackage{amsthm}
\usepackage{graphicx}
\usepackage{fancyhdr}
\usepackage{bbm}
\usepackage{pgfplots}

\newfont{\bssten}{cmssbx10}
\newfont{\bssnine}{cmssbx10 scaled 900}
\newfont{\bssdoz}{cmssbx10 scaled 1200}

\usepackage{latexsym,amssymb,amsthm,amsxtra}
\usepackage{amsmath}
\usepackage{stmaryrd}
\usepackage{mathrsfs}
\usepackage{tikz}
\usepackage{pgf}
{\bf}{\it}
{\bf}{\it}
\newtheorem{theorem}{Theorem}
\newtheorem{definition}{Definition}
\newtheorem{lemma}{Lemma}
\newtheorem{remark}{Remark}
\newtheorem{proposition}{Proposition}
\newtheorem{corollary}{Corollary}

\def\Z{{\mathbb Z}}

\def\ra{\longrightarrow} 

\def\bp{\mathbf{p}}
\def\bx{\mathbf x}
\def\by{\mathbf y}
\def\ba{\mathbf a}
\def\bb{\mathbf b}
\def\bs{\mathbf s}

\def\sb{\text{\textsc{b}}}

\def\sf{\text{\textsc{f}}}

\def\ni{\text{\textsc{ni}}}

\def\esp#1{{\mathbb E}\left[#1\right]}

\newcommand{\pae}[1]{\mbox{$\lfloor \kern-1pt #1 \kern-1pt \rfloor$}}
\newcommand{\paep}[1]{\mbox{$\lceil \kern-1pt #1 \kern-1pt \rceil$}}

\def\R{{\mathbb R}}
\def\maR{{\mathcal R}}
\newcommand{\ord}[1]{\maR\left(#1\right)}
\newcommand{\pos}[1]{n^{\small{+}}#1}

\newcommand\gre{\textbf{e}}

\newcommand\maA{{\mathcal A}}

\newcommand\maS{{\mathcal S}}

\def\ohad#1{\textcolor{red}{#1}}
\def\pasc#1{\textcolor{blue}{#1}}

\newcommand{\af}{\alpha}
\newcommand{\lm}{\lambda}
\newcommand{\ep}{\epsilon}

\newcommand{\deq}{\stackrel{\rm d}{=}}

\newcommand{\qandq}{\mbox{ and }}

\newcommand{\qiffq}{\quad\mbox{if and only if}\quad}

\newcommand{\qforallq}{\quad\mbox{for all }}

\newcommand{\RR}{{\mathbb R}}
\newcommand{\ZZ}{{\mathbb Z}}

\def\tinf{\rightarrow\infty}

\newcommand{\be}{\begin{equation}}
\newcommand{\bes}{\begin{equation*}}
\newcommand{\ee}{\end{equation}}
\newcommand{\ees}{\end{equation*}}
\newcommand{\bequ}{\begin{equation}}
\newcommand{\eeq}{\end{equation}}
\newcommand{\bsplit}{\begin{split}}
\newcommand{\esplit}{\end{split}}
\newcommand{\bea}{\begin{eqnarray}}
\newcommand{\eea}{\end{eqnarray}}
\newcommand{\beas}{\begin{eqnarray*}}
\newcommand{\eeas}{\end{eqnarray*}}

\def\Vcr{V_{\text{cr}}}
\def\gsc{\preceq_{\textsc{gsc}}}

\title{Stability of Parallel Server Systems}
\author{Pascal Moyal and Ohad Perry}
\begin{document}

\begin{abstract}
The fundamental problem in the study of parallel-server systems is that of finding and analyzing ``good'' routing policies of arriving jobs to the servers.
It is well known that, if full information regarding the workload process is available to a central dispatcher, then the {\em join the shortest workload} (JSW) policy,
which assigns jobs to the server with the least workload, is the optimal assignment policy, in that it maximizes server utilization, and thus minimizes sojourn times.
The {\em join the shortest queue} (JSQ) policy is an efficient dispatching policy when information is available only on the number of jobs with each of the servers,
but not on their service requirements.
If information on the state of the system is not available, other dispatching policies need to be employed, such as the power-of-$d$ routing policy,
in which each arriving job joins the shortest among $d \ge 1$ queues
sampled uniformly at random. (Under this latter policy, the system is known as {\em the supermarket model}.)
In this paper we study the stability question of parallel server systems assuming that routing errors occur,
so that arrivals may be routed to the ``wrong'' (not to the smallest) queue with a positive probability.
We show that, even if a ``non-idling'' dispatching policy is employed, under which new arrivals are always routed to an idle server, if any is available,
the performance of the system can be much worse than under the policy that chooses one of the servers uniformly at random.
More specifically, we prove that the usual traffic intensity $\rho < 1$ does not guarantee that the system is stable.
\end{abstract}

\maketitle

\section{Introduction} \label{secIntro}

We consider a parallel-server system with $s \ge 2$ statistically-homogeneous servers, each providing service at rate $\mu$,
that is fed by a rate-$\lm$ Poisson arrival process of statistically identical jobs (or customers). For each server there is a dedicated infinite buffer in which jobs queue, waiting for their turn to be served.
Upon arrival, a job is routed to one of the $s$ servers according to some pre-specified dispatching (routing) rule, with no jockeying between the queues allowed.
In this setting, one seeks a ``good'' routing policy of jobs to the servers, e.g., a policy ensuring that steady state waiting times are minimized, or that the total throughput rate is maximized.
If the workload at each queue can be computed, then it is natural to employ the Join the Shortest Workload (JSW) routing policy, under which an arriving job is routed to
the server with the least workload among all $s$ servers (together with some tie-breaking rule). However, if the workload is unknown, as is often the case in practice,
one may opt to employ the Join-the-Shortest Queue (JSQ) control, which routes an arriving job to the server with the smallest number of jobs.
Indeed, JSW was shown to minimize the workload process in \cite{DaleyOptimal}, whereas
JSQ has been shown to be throughput maximizing in terms of stochastic order, when the service-time distribution has a non-decreasing failure rate \cite{weber78},
and in particular, when the service times are exponentially distributed \cite{Winston77}.

However, even the queue at each server is not always known: In some settings, the number of customers in each queue is estimated, either by the arriving customers who are free to choose
which queue to join (as in a supermarket or security lanes in airports), or by a central dispatcher (as is often the case in passport-checking stations, for example).
Even in automated settings the queue lengths may not be known. For example, information regarding the queues to each of the servers in web-server farms requires constant communication between the servers and the job dispatchers,
slowing down the response time, and is thus not always available; e.g., see \cite{Lu11}.

For this reason, other routing policies have been considered in the literature, most notably the ``power-of-$d$''
policy, which gives rise to the so-called ``supermarket model'' \cite{Mitzenmacher96}. Under this policy, upon each arrival
$d$ servers are chosen uniformly at random, and that arrival is routed to the server with the smallest number of jobs among the $d$ sampled queues, with ties broken uniformly at random.
We denote this routing rule by PW($d$) and note that $d=1$ corresponds to uniform routing (i.e. any incoming job is sent to a queue that is chosen uniformly at random),
whereas $d=s$ corresponds to JSQ.

\subsection{Motivation and Goals} \label{secMotivation}
We are motivated by the fact that, and unlike the idealized settings considered in the literature, routing errors can occur in practice, so that jobs are not always routed in an efficient manner.
In this regard, our main goal is to demonstrate that routing errors can have substantial negative impacts on performance. To this end, we study a particular form of error,
under which arrivals are sent to the ``wrong'' queue (not the smallest) with a fixed probability, and show that the system might not be stable in this case, even if its total service rate is larger than
the rate at which work arrives, i.e., if the traffic intensity to the system is smaller than $1$.

Such errors are likely to occur when JSW is employed, because the actual workload at each server can only be estimated, unless the server is idle (in which case its workload is zero),
but can also occur under JSQ, especially when there is no central dispatcher, and customers choose which queue to join.
We focus on the latter JSQ policy, since under appropriate distributional assumptions (Poisson arrival process and exponentially distributed service times), the queue process evolves as a
continuous-time Markov chain (CTMC), whereas under JSW, the analysis of the queue process requires a continuous-space Markov representation. (Even under JSQ, exact analyses and steady-state computations of the queue
are intractable, and most of the literature is concerned with asymptotic approximations; see Section \ref{secLit} below.)
It will become intuitively clear, and supported by simulation examples in Section \ref{sec:simu}, that our results extend to the JSW case.

Even though our main motivation is to study the impact of routing errors, we treat the allocation of jobs to servers as a routing policy.
We do this for mathematical convenience, as it allows us to treat PW($d$), and therefore also JSQ, as a special case of the family of allocation policies we consider.
Specifically, we assume that the dispatcher (or the arriving customer) chooses correctly the shortest queue with probability $p_1$,
the second-shortest queue with probability $p_2$, and so forth.
We also consider a {\bf ``non-idling'' case}, in which routing errors are made only when all servers are busy, so that the dispatcher (or arriving customer) always chooses an idle server, if such a server is available,
and otherwise makes errors as was just described.
To show that such errors can lead to extreme departures from the desired behavior under JSQ,
we characterize the stability region under the allocation policy as a function of the system's parameters and the error probabilities,
and prove that the usual traffic condition $\rho := \lm/(s\mu) < 1$ does not guarantee that the system is stable, even in the non-idling case.

\subsection{Background: PW($d$) and Related Routing Policies}
Note that it is not immediately clear that the condition $\rho < 1$ does not imply that the system is stable, especially under the non-idling allocation mechanism,
because the JSQ policy (and of course, JSW) leaves a lot of ``room'' for making routing errors, as can be seen by comparing a system operating under JSQ to the same system operating under uniform routing.
Clearly, uniform routing induces a lot of ``avoidable'' idleness in the system, because arrivals are often routed to busy servers even if there are idle servers present.
Nevertheless, by symmetry, the rate at which jobs arrive at each server is the same under this policy, implying that the traffic intensity at each server separately is smaller than $1$ whenever the traffic intensity $\rho$ to the
whole system is smaller than $1$.
When the arrival process to the system is Poisson, this follows directly from the splitting property of the Poisson process, which implies that each server operates as an $M/G/1$
queue independently of all other servers.
Indeed, if service times are exponentially distributed, in addition to having a Poisson arrival process, so that
the queue process evolves as a CTMC, the improvement that JSQ provides over uniform routing follows from existing results, which we now review.

Let $Q_\Sigma^{(d)}(t)$ denote the total number of jobs in the system at time $t \ge 0$ under PW($d$). Theorem 4 in \cite{Turner98} implies that\footnote{Theorem 4 in \cite{Turner98}
proves a  monotone convex order domination, from which sample-path stochastic order follows immediately}, if $d_1 > d_2$, then $Q_\Sigma^{(d_1)} \le_{st} Q_\Sigma^{(d_2)}$,
where $\le_{st}$ denotes sample-path stochastic-order.
(That is, there exists a coupling of the two processes, such that $Q_\Sigma^{(d_1)}(t) \le Q_\Sigma^{(d_2)}(t)$ w.p.1 for all $t > 0$, provided that the inequality holds at time $t = 0$.)
In particular, for $s > 2$,
\be
\label{eq:ineqPW}
Q_\Sigma^{(s)} \le_{st} Q_\Sigma^{(d)} \le_{st} Q_\Sigma^{(1)}, \quad 1 < d \le s.
\ee
The stability of a parallel-server system under PW($d$) readily follows. To state this result formally, we say that a parallel-server system is ``Markovian'' if its multi-dimensional queue process
evolves as a CTMC. In particular, the arrival process is Poisson and the service times are independent and identically distributed
(i.i.d.) exponentially distributed random variables, that are independent of the arrival process and of the state of the system.
\begin{corollary} \label{coroStable}
For a Markovian parallel-server system with $s$ servers operating under PW($d$), $1 \le d \le s$, the condition $\rho := \lm/(s\mu) < 1$ is necessary and sufficient in order for the queue process to be an ergodic CTMC.
\end{corollary}

\begin{proof}
It is easy to see that $Q_\Sigma^{(d)}$ is an irreducible CTMC.
If $\rho \ge 1$, then $Q_\Sigma^{(d)}$ is either null recurrent or transient, because it is bounded from below, in sample-path stochastic order,
by the number-in-system process in an $M/M/1$ queue with arrival rate $\lm$ and service rate $s\mu$.
On the other hand, if $\rho < 1$, then $Q_\Sigma^{(1)}$ is ergodic, because it evolves as
$s$ independent $M/M/1$ queues, each with arrival rate $\lm/s$ and service rate $\mu$. In particular the empty state (zeroth vector) is positive recurrent for the CTMC $Q_\Sigma^{(1)}$,
and, by virtue of \eqref{eq:ineqPW}, also for $Q_\Sigma^{(d)}$, $1 < d \le s$.
\end{proof}

A more quantitative analysis can be carried out asymptotically, by taking the number of servers $s$ to infinity, assuming that the arrival rate grows proportionally to $s$.
As was shown in \cite{Mitzenmacher96, Vved96}, the steady-state probability that an arrival is routed to a queue of length at least $k$ is $\rho^{d^k}$, i.e., it is doubly exponential in $k$ for $d \ge 2$,
as opposed to exponential when $d = 1$ (which is tantamount to uniform routing).
The dramatic differences between the {\em maximum} queue length in stationarity in the cases $d=1$ and $d \ge 2$
is demonstrated in \cite{Luczak06}, which shows that the maximum queue length is of order $\ln(s) / \ln(1/\lm)$ when $d=1$, and of order $\ln\ln(s)/\ln(d)$ when $d \ge 2$ with probability converging to $1$ as
$s \ra \infty$.
Further, heavy-traffic analysis shows that the performance under PW($d$), for any fixed $d < s$, is substantially worse than under JSQ.
In particular, considering a sequence of systems indexed by the number of servers $s$, and letting $\lm_s$ denote the arrival rate to system $s$,
\cite{Gamarnik_JSQ} and \cite{Gamarnik_supermarket} analyze a system operating under JSQ and PW($d$), respectively, in the heavy-traffic limiting regime, where $\lm_s = s\mu - \Theta(\sqrt{s})$.
It is proved in \cite{Gamarnik_JSQ} that, under JSQ, only a negligible proportion (which converges to $0$) of the customers encounter a queue upon arrival, and those customers that have to wait encounter only
one customer in queue. Thus, asymptotically, no queue is larger than $2$. (This result holds only after some transient period, because the initial condition may have many larger queues.)
On the other hand, \cite{Gamarnik_supermarket} proves that, in the supermarket model with $d > 1$, the fraction of queues that are of order $\log_d \sqrt{s}$ approaches $1$ as $s \tinf$.

To conclude, the dimensionality of the queue process, and the fact that it is not reversible, render exact analysis of parallel-server systems intractable, even
under Markovian assumptions. Other than stability results and stochastic domination, as in \eqref{eq:ineqPW},
little can be said about the systems' dynamics and steady state distributions. Nevertheless,
the aforementioned asymptotic results suggest that JSQ is substantially more efficient than
PW($d$) for $d < s$, which, in turn, is substantially more efficient than uniform routing, namely, than PW($1$).

\subsection{Notation}
\label{subsec:not}
We use $\R$ to denote the set of real numbers, with $\R_+ = [0, \infty)$, $\Z_+$ to denote the set of non-negative integers, and $\Z_+^* := \ZZ_+ - \{0\}$ the subset of (strictly) positive integers.
For any $q \in \Z_+$ and all sets $A$, we denote by
$A^q$ the set of vectors of dimension $q$ having elements in $A$, e.g., $\RR^q$ is the set of $q$-dimensional real-valued vectors.
Vectors are in general denoted by bold letters.
For a vector $\bx=(x_1,...,x_q)$ in $\R^q$, we denote by $\ord{\bx}$ the ordered version of $\bx$, i.e. $\ord{\bx} = (x_{(1)}, x_{(2)},\dots, x_{(q)})$ is any permutation of the elements of $\bx$ such that
$x_{(1)} \le x_{(2)} \le \cdots \le x_{(q)}$.
vector $\ord{\bx}$.)
The set of ordered vectors in $A^q$ is denoted by $\ord{A^q}$; for example, $\ord{\R_+^q} := \{\bx \in \RR^s_+ : x_1 \le \cdots \le x_q\}$.

We let $\ba \circ \bx \in \R^q$ denote the Hadamard product of two vectors $\bx=(x_1,...,x_q)$ and $\by=(y_1,...,y_q)$ in $\R^q$, i.e., $\by \circ\bx=(y_1x_1,...,y_q x_q)$.
For $\bx \in \R_+^q$, we define $\pos(\bx)$ to be the number of positive coordinates of $\bx$, which is $0$ if $\bx$ is the zeroth vector $\mathbf 0 := (0,\dots, 0)$.
Let $\llbracket p,q \rrbracket = \Z_+ \cap [p,q]$. For any $i \in \llbracket 1,q \rrbracket$, let $\gre_i$ denote the vector having all coordinates null except the $i$th one, equal
to $1$, and let $\gre$ denote the unit vector whose components are all equal $1$; $\gre := (1,\dots, 1)$.
For any $\bx \in \R_+^q$ we denote by $\parallel \bx \parallel = \sum_{i=1}^q x_i$ and $\parallel \bx \parallel_2 = \sum_{i=1}^q (x_i)^2$.
For any two real numbers $a$ and $b$, let $a \vee b$ and $a \wedge b$ denote the maximum and the minimum of $a$ and $b$, respectively. Let $a^+= a \vee 0$.

\subsection{Organization}
The rest of the paper is organized as follows: We provide a detailed literature review in Section \ref{secLit}.
The model, including the family of allocation policies, which we call $\bp$-allocation policies, is formally introduced in Section \ref{sec:model}.
In Section \ref{sec:suff} we study a class of $\bp$-allocation policies for which
the condition $\rho<1$ implies that the system is stable. The insufficiency of this traffic condition to imply stability in general is demonstrated in Section \ref{secInsufficiency}.
In Section \ref{sec:simu} we present simulation results suggesting that our main results extend to workload-based routing policies.
The proofs of two technical lemmas appear in Appendix \ref{sec:proofs}, and two additional supporting results appear in Appendix \ref{sec:aux}.

\section{Related Literature} \label{secLit}

\paragraph{{\em Non-monotonic parallel queues.}}
Under JSW, the dynamics of the system, as well as the sojourn time of jobs, coincide with those of a single-queue $s$-server system operating under the First In First Out (FIFO) service policy.
In particular, that $\rho < 1$ is a necessary and sufficient condition for the stability of the system under JSW follows from from the basic stability theory of the $GI/GI/s$ queue,
first proved in the seminal paper \cite{KW55}.
The sufficiency of the condition $\rho < 1$ for stability of the $G/G/s$ queue was generalized in \cite{Brandt85} to the stationary ergodic framework, namely,
when both the inter-arrival and service-time sequences are time-stationary and ergodic,
but not necessarily independent; see also \S 2.2 of \cite{BacBre02}. This general result was proved using a backwards scheme
of the Loynes type \cite{Loynes62}, building on the fact that the (random) updating map of the stochastic recursive sequence representing the system is non-decreasing for the coordinate-wise vector ordering.
For the same reason, JSW is the unique routing rule within the class of semi-cyclic policies introduced in \cite{SW03}, which renders the total workload to be a non-decreasing function of $s$ at all times; see \cite{Moy17_b}.
Therefore, the stability region under allocation policies {\em other than JSW} cannot simply be characterized via a Loynes-type construction,
and we must therefore adopt a different approach.

\paragraph{{\em JSQ systems.}}
The JSQ policy was first introduced in \cite{Haight58} for a system with two servers, each having a different service rate.
The first proof that the condition $\rho < 1$ is necessary and sufficient for
a Markovian parallel-server system under JSQ to be stable (admit a steady state) appears in \cite[Theorem 1]{King61} for a system with $s=2$ servers, building on
a straightforward Lyapunov stability argument. The main goal of \cite{King61} is to characterize the stationary distribution of the (stable) system via generating functions.
Explicit computation of this distribution is provided in \cite{FMcK77}. A systems with finite buffers is studied in \cite{DFT17} which provides closed-form expressions for the loss probabilities.
A non-idling version of JSQ was proposed and analyzed in \cite{Lu11} which considers systems with more than one dispatcher, and analyzes how to balance information regarding idle servers among those dispatchers.

There are several papers that study JSQ in asymptotic regimes. In addition to \cite{Gamarnik_JSQ}, which was discussed above, we mention \cite{Gra00}, which
identifies a mean-field limit, and shows the chaoticity of the system as $N$ increases.
An Orstein-Uhlenbeck limit for the same model is obtained in \cite{Gra05}.

In general, Lyapunov-stability arguments, as in \cite{King61}, can be hard to generalize to higher-dimensions, because of the need to control the drifts of the process
at all states outside some compact subset of the state space.
Our proof of Theorem \ref{thMaximal} below, that $\rho < 1$ implies that the system is stable for a certain subset of control parameters, is a generalization
of \cite[Theorem]{King61}, both because it allows any number of servers $s$, and because it considers a larger family of routing policies, for which JSQ is a special case.
In the latter regard, it also generalizes Corollary \ref{coroStable}.
Our proof is achieved by employing a certain partial-order relation (see Definition \ref{defOrder} in Section \ref{subsec:order}) in conjunction with a Lyapunov-stability argument.

\paragraph{{\em Power-of-$d$ allocations.}}
The PW($d$) policy was first studied in \cite{Vved96} and \cite{Mitzenmacher96}, which also coined the term ``supermarket model'' to describe a system operating under this control.
The supermarket model has since received substantial attention due to its practical and theoretical significance.
Both \cite{Gamarnik_supermarket} and \cite{Brightwell12} study the supermarket model in heavy traffic, namely, as the traffic intensity approaches $1$.
The rate at which the equilibrium distribution of a typical queue converges to the limiting one in the total-variation distance is studied in \cite{Luczak07}, which also quantifies the chaotic behavior
of the system, asymptotically, namely, the rate at which the joint distribution of any fixed number of queues converges to the limiting product-form distribution.
Finally, we mention a recent game-theoretic supermarket model in \cite{SupermarketGame}, which is also analyzed asymptotically, as the number of servers and arrival rate increase to infinity.

It is significant that the asymptotic result regarding the doubly exponential decay rate of the queue size in equilibrium does not necessarily hold for general service-time distributions.
Indeed, \cite{Bramson10} shows that, for some power-law service-time distributions, the equilibrium queue sizes decay at an exponential, or even polynomial, rate,
depending on the power-law exponent and the number of sampled queues $d$.

\paragraph{{\em Robustness of Control.}}
The dynamics of a system under a given control are typically studied in idealized settings, which do not fully hold in practice. In particular, even small deviations from the theoretical
implementation of a control (due to, e.g., human or measurement errors, discretization of a continuous control process, delays in making or applying a decision, etc.),
can in turn lead to substantial perturbations from theoretically predicted performance.
Such discrepancies between theory and implementation constitute an important area of research in dynamical control theory (see, e.g., \cite[\S 14]{khalil02} and \cite{Liberzon03}),
but received little attention in the queueing literature.
In \cite{PWchatter} it is shown how the implementation of a control, that has theoretically desirable performance in a certain asymptotic regime, can lead
to chattering of the queue process and, in turn, to {\em congestion collapse}, namely, to a severe overload that is solely due to the implementation of the control.
We refer to \cite[Section 9]{PWchatter} for a detailed, albeit informal, discussion on how small perturbations from idealized control settings can have substantial impacts on the performance
of queueing systems.

\paragraph{{\em Instability of Subcritical Systems.}}
Congestion collapse is related to the more general research area regarding instability of subcritical networks, which initialized with the presentation of
the (deterministic) Lu-Kumar network studied in \cite{LuKumar}, and its stochastic counterpart, the Rybko-Stolyar network \cite{RySto92}; see also \cite[\S]{BramsonBook} and
\cite{MoyalPerry} for applications and literature reviews.
A non-idling policy is considered in \cite{Moy17_a}, in which an arrival is routed to the queue having the $2$nd shortest workload.
A sufficient condition for stability, that is strictly stronger than $\rho < 1$, is provided,
and it is conjectured that the latter condition is also necessary.

In ending we remark that the possibility of experiencing congestion collapse in parallel-server systems can be considered a triviality for vacuous choices of the control.
For example, if the arrival rate $\lm$ is larger than the service rate $\mu$ (but is smaller than $s\mu$), then the policy that routes all arrival to the same server is clearly unstable.
Here, however, we perform a refined analysis of the (in)stability region for the non-idling version of JSQ when routing errors occur with a nonnegligible probability.

\vspace{.5cm}
\section{The Model} \label{sec:model}

We consider the following class of parallel systems: There are $s$ servers, each having its own infinite buffer for waiting jobs.
Jobs arrive to the system following an homogeneous Poisson process with intensity $\lm$, and join one of the servers according to a routing policy from a class of policies that will be formally defined immediately.
If the server to which a job is routed is idle, that job enters service immediately;
otherwise, it joins the end of the server's dedicated queue, waiting for its turn to be served (there is no jockeying between queues).
All jobs are statistically equivalent, requiring
i.i.d.\ service times that are exponentially distributed with mean $1/\mu$, regardless of the server.
We let $\rho := \lm/(s\mu)$ denote the traffic intensity to the system.

Even though this routing mechanism is an erroneous execution of JSQ, we treat it as a control,
which we call a ``$\bp$-allocation policy'', where $\bp$ is the {\em allocation probability vector} $\bp=(p_1,p_2,...,p_s)$.
With this view, the PW($d$), and in particular, JSQ and uniform splitting,
become special cases of the $\bp$-allocation policy; see \eqref{eq:pU}-\eqref{eq:pPWd} below.

The class of allocation policies we consider depend only on the queue sizes (number of customers in service plus the number of customers waiting in line) of the servers.
To determine the server allocations without ambiguity, we assume that the servers
are re-labeled as $1,2,...,s$ upon each event (arrival or departure), such that $i < j$ if the queue size for server $i$ is no larger than the queue for server $j$.
Servers having the same queue size have consecutive labels; the labeling within each such group of servers can be arbitrary, but for concreteness, we assume that it is made uniformly at random.
Therefore, with $Q_i(t)$ denoting the queue size of server $i$ at time $t \ge 0$,
the vector $Q(t):=\left(Q_1(t),...,Q_s(t)\right)$ is an element of $\ord{\Z_+^s}$.
We let $Q_{\Sigma}(t)=\sum_{i=1}^d Q_i(t)$ denote the total number of customers in the system at time $t$.

Let $\Pi^s$ denote the family of probability vectors on $[0,1]^s$, namely, a vector $\bp := (p_1, \dots, p_s)$ is in $\Pi^s$ if $p_i \in [0,1]$, $1 \le i \le s$, and $\sum_{i=1}^s p_i = 1$.
\begin{definition}
We call a routing policy a {\bf $\bp$-allocation policy}, and call $\bp$ the {\bf allocation (probability) vector}, $\bp \in \Pi^s$,
if, upon arrival, a customer is sent to server $i$ with probability $p_i$, independently of everything else.
A $\bp$-allocation policy is said to be {\em non-idling} if an incoming job is routed to an idle server, whenever there is one upon that job's arrival,
and is otherwise routed to server $i$ with probability $p_i$, independently of everything else.
\end{definition}
In particular, for each $\bp$-allocation policy there is a corresponding non-idling version which uses the same allocation vector to route jobs that arrive when all servers are busy,
and otherwise route the arrivals to one of the idle servers.

Observe that if two or more queues are equal upon an arrival, a $\bp$-allocation policy
assigns the incoming customer to one of those queues with an equal probability. Indeed, if a customer enters the system at $t$ and the consecutive indices $j,j+1,...,k-1,k$ are such that $Q_{j-1}(t^-)<Q_j(t^-)=Q_{j+1}(t^-)=....Q_{k-1}(t-)=Q_{k}(t-)<Q_{k+1}(t-)$, 
then by uniformity of the choice of labeling, server $\ell$ is chosen with the probability
$${1 \over k-j+1}\sum_{i=j}^k p_i, \quad \text{for any $\ell \in \llbracket j,k \rrbracket$.}$$

A particular class of $\bp$-allocation policies is the PW($d$) policy, and its special cases, uniform splitting and JSQ.
\begin{itemize}
\item For uniform splitting, the allocation vector is
\begin{equation}
\label{eq:pU}
\bp^{(1)}:=\left(1/s,...,1/s\right).
\end{equation}
\item For JSQ, we have
\begin{equation}
\label{eq:pJSQ}
\bp^{(s)}:=(1,0,...,0).
\end{equation} 
\item More generally, under PW($d$) an arriving job is routed to server $i$ if it is one of the $d$ draws,
and the other $d-1$ servers drawn have indices in $\llbracket i+1,d \rrbracket$.
Then the allocation vector for this policy is (with ties broken uniformly at random)
\begin{equation}
\label{eq:pPWd}
\bp^{(d)}:=\left(p^{(d)}_1,...,p^{(d)}_s\right)
=\begin{cases}
p^{(d)}_i ={s-i \choose d-1} / {s \choose d},&\quad i\in\{1,...,s-d+1\};\\
p^{(d)}_i =0,&\quad i\in \{s-d+2,\dots,s\},
\end{cases}\end{equation}
\end{itemize}
Observe that \eqref{eq:pU} and \eqref{eq:pJSQ} are consistent with \eqref{eq:pPWd}, and are achieved by taking $d=1$ and $d=s$, respectively.

\subsection{The Stability Regions of the Allocation Policies}
It is immediate that for any probability vector $\bp \in \Pi^s$, the process $Q$ is an $\ord{\Z_+^s}$-valued continuous-time Markov chain (CTMC).
The {\em stability region} of the parallel-server system corresponding to the $\bp$-allocation policy, which we denote by $\mathcal S(\bp)$,
is then defined as the set of values of the traffic intensity $\rho = \lm/(s\mu)$ under which $Q$
is stable in the sense that it is a positive recurrent CTMC. Then for any $\bp$-allocation vector we define
\begin{align*}
\mathcal S(\bp) &:= \left\{\rho \in [0,1)\,:\, \mbox{$Q$ is positive recurrent under the $\bp$-allocation policy}\right\};\\
\mathcal S^\ni(\bp) &:= \left\{\rho \in [0,1)\,:\, \mbox{$Q$ is positive recurrent under the {\bf non-idling} $\bp$-allocation policy}\right\}.
\end{align*}

It is intuitively clear that the stability region under a non-idling $\bp$-allocation policy cannot be smaller than the stability region under the same allocation vector when the policy is not non-idling.
This is formally proved in the next proposition.
\begin{proposition}
\label{prop:busyvsidling}
$\maS(\bp) \subseteq \maS^\ni(\bp)$ for all $\bp \in \Pi^s$.
\end{proposition}

\begin{proof}
Consider an allocation vector $\bp$ together with an arrival rate $\lm$ and service rate $\mu$, such that $\rho \in \maS(\bp)$,
and the corresponding queue process $Q$. Observe that the traffic intensity $\rho$ is then necessarily less than 1. Denote by $Q^\ni$ the queue process in the system operating under the corresponding non-idling $\bp$-allocation policy,
and by $Q^{(s)}$ the queue process of a system of same traffic load, operating under the JSQ policy (equivalently, under the PW($s$) policy).
It is easily seen that the process $Q$ coincides in distribution with the process $Q^{\ni}$ on the subset $F := \{\bx \in \RR^s : x_i \ge 1, i \in \llbracket 1, s \rrbracket\}$ of the state space,
and with $Q^{(s)}$ on the complement subset $F^c := \{\bx \in \RR^s : \bx \notin F\}$.
The result follows from the fact that the process $Q^\ni$ is ergodic by assumption, together with the fact that the process $Q^{(s)}$ is ergodic for any $\rho < 1$ due to \eqref{eq:ineqPW}.
(Recall that $Q_\Sigma^{(1)}$ in \eqref{eq:ineqPW} is the queue under uniform splitting, which operates like $s$ independent $M/M/1$ queues, each with traffic intensity $\lm/\mu < 1$.)
\end{proof}

\begin{remark}
It is significant that $\maS(\bp) \ne \maS^\ni(\bp)$ in general; in particular, there exist $\bp$-allocation policies for which $\maS^\ni(\bp)$ is strictly larger than $\maS(\bp)$.
To see why the proof of Proposition \ref{prop:busyvsidling} cannot be adapted to show the containment in the other direction
(i.e., to show that $\maS^\ni(\bp) \subseteq \maS(\bp)$), consider a $\rho$ for which $Q^\ni$ is stable under some $\bp$-allocation policy.
Note that, if $Q$ is not known to be an ergodic CTMC at the outset, then there is no guarantee that the expected hitting time of the set $F$ by the process $Q$ is finite,
or even that this hitting time is finite w.p.1. Therefore, even though the expected hitting time of $F^c$ by $Q$ is finite, because $Q^\ni$ is assumed
to be ergodic and $Q$ is locally distributed the same as $Q^\ni$ while in $F$, it is possible that the process $Q$ is absorbed in $F^c$.
\end{remark}

As an immediate consequence of Proposition \ref{prop:busyvsidling} we see that, if stability is proved for given system's parameters and for a specific $\bp$-allocation policy (a specific allocation vector $\bp$),
then the system is also stable under the non-idling version of that policy.
On the other hand, a system is unstable if operated under a $\bp$-allocation policy, if it is shown to be unstable under its non-idling version.

\vspace{.5cm}
\section{Maximal $\bp$-Allocation Policies} \label{sec:suff}
In this section we identify a sub-class of $\bp$-allocation policies under which the stability region is the interval $[0,1)$.
We call such an allocation policy {\em maximal}, since its stability region is the largest possible.

\subsection{Preliminary}
\label{subsec:order}
We will state a sufficient condition on the $\bp$-allocation probability that ensures that the system is stable if $\rho < 1$.
That condition is expressed in terms of the following partial order on $\R^s_+$.
\begin{definition} \label{defOrder}
Let $\ba = (a_1,...,a_s)$ and $\bb=(b_1,...,b_s)$ be elements of $\R_+^s$, $s\ge 1$.
We say that $\ba$ is smaller than $\bb$ in the ``generalized Schur-convex'' order, and write $\ba \preceq_{\textsc{gsc}} \bb$, if
\[\sum_{i=k}^s a_i \le \sum_{i=k}^s b_i\mbox{ for all }k \le s.\]
\end{definition}
The relation ``$\preceq_{\textsc{gsc}}$" defines a partial ordering on $\R_+^s$ that is a variant (for non-necessarily ordered vectors)
of the partial semi-ordering ``$\prec_{\textsc{cx}}$" introduced in Definition 3 of \cite{Moy17_b}, which itself generalizes the well-known Schur-convex partial semi-ordering ``$\prec_{\textsc{scx}}$"
(see e.g. \cite{MO79}) to vectors of different total sums. Specifically, we have $\ba \preceq_{\textsc{gsc}} \bb$ if and only if $\ba \prec_{\textsc{cx}} \bb$ for any $\ba,\bb \in \ord{\R_+^s}$,
and $\ba \preceq_{\textsc{gsc}} \bb$ if and only if $\ba \prec_{\textsc{scx}} \bb$ for any $\ba,\bb \in \ord{\R_+^s}$ such that $\parallel \ba \parallel = \parallel \bb \parallel$.

Observe that, for any random variables $X$ and $Y$ having respective probability mass functions $\bp_X$ and $\bp_Y$ in $\Pi^s$ and values in
$\llbracket 1,s \rrbracket $, it holds that
$X \le_{st} Y$ if and only if $\bp_X \preceq_{\textsc{gsc}} \bp_Y.$
The following monotonicity result is proved in appendix,
\begin{lemma}
\label{lemma:order}
Let $\ba$ and $\bb$ be two vectors in $\R_+^s$ such that $\ba \preceq_{\textsc{gsc}} \bb$, and let $\bx \in \ord{\R_+^s}$.
Then
$$\bx \circ \ba \preceq_{\textsc{gsc}} \bx \circ \bb.$$
\end{lemma}

\subsection{A Sufficient Condition for Stability}\label{subsec:maximal}
The main result of this section shows that if, in addition to $\rho < 1$, it holds that the $\bp$-allocation probability vector is no
larger, in the $\gsc$ order, than the uniform probability on $\llbracket 1,s \rrbracket$, namely, if $\bp\in\Pi^s$ satisfies
\bequ \label{Ncond}
\bp \gsc \bp^{(1)},
\eeq
for $\bp^{(1)}$ in (\ref{eq:pU}), then the system is stable.
\begin{theorem}\label{thMaximal}
If $\bp$ satisfies \eqref{Ncond}, then $\maS(\bp) = [0,1)$, namely, the $\bp$-allocation policy is maximal.
\end{theorem}

\begin{proof}
For $n \ge 0$, let $T_n$ denote the $n$th transition epoch of the CTMC $Q$, with $T_0 = 0$, and consider the embedded discrete-time Markov chain (DTMC) $\{Q_n : n \ge 0\}$ defined via
$Q_n := Q\left(T_n\right)$. We prove the result via a Lyapunov stability argument, employing the Lyapunov function $V: \ord{\Z_+^s} \longrightarrow \R+$ defined by $V(x) = \| \bx \|_2$.
Let
\[\mathcal K = \left\{\bx \in \ord{\Z_+^s}\,:\, \sum_{i=1}^s x_i \le {s(\lambda + s\mu) \over 2(s\mu-\lambda)} \right\}.\]

Then, for any $n \ge 1$ and
$\bx=(x_1,...,x_s) \in \mathcal K^c \cap \ord{\Z_+^s}$ we have
\begin{multline} \label{eqLyp1}
\esp{V\left(Q_{n+1}\right) - V\left(Q_n\right) \mid Q_n = \bx}\\
\begin{aligned}
&= \sum_{i=1}^s {\lambda \over \lambda + \pos(\bx)\mu}p_i\left((x_i + 1)^2 - (x_i)^2\right)
+ \sum_{i=1}^s {\mu \over \lambda +\pos(\bx)\mu}\left(((x_i - 1)^+)^2 - (x_i)^2\right)\\
&= {1 \over \lambda + \pos(\bx)\mu}\left(2\left(\lambda\sum_{i=1}^s p_ix_i-\mu\sum_{i=1}^s x_i\right) + \lambda +\pos(\bx)\mu \right).
\end{aligned}\end{multline}
Applying Lemma \ref{lemma:order} with $\ba := \bp$, $\bb := \bp^{(1)}$, for $\bp^{(1)}$ in (\ref{eq:pU}), and the ordered vector $\bx$, we obtain that $\bx \circ\bp \preceq_{\textsc{gsc}} \bx\circ \bp^{(1)}$,
and in particular that $\sum_{i=1}^s p_ix_i \le {1 \over s}\sum_{i=1}^s x_i.$ As $\pos(\bx) \le s$, this entails that the last expression in \eqref{eqLyp1} is less than or equal to
\[{1 \over \lambda + \pos(\bx)\mu}\left(2\left({\lambda \over s}-\mu\right)\sum_{i=1}^s x_i + \lambda +s\mu \right),\]
which is strictly negative for $\bx \notin \mathcal K$.
In particular, for all $\bx=(x_1,...,x_s) \in \mathcal K^c \cap \ord{\Z_+^s}$ and all $n$,
\[\esp{V\left(Q_{n+1}\right) - V\left(Q_n\right) \mid Q_n = \bx}<0.\]
We deduce from the Lyapunov-Foster Theorem (see, e.g., \cite[\S 5.1]{Bre99}) that the DTMC $\{Q_n : n \ge 1\}$ is positive recurrent. In turn, this implies that the CTMC $Q$ is positive recurrent as well,
by Theorem 6.18 in \cite{Kulkarni17}, as the rate of the exponentially distributed holding time in each of the states is bounded from below by $\lm$.
\end{proof}

As discussed in Section \ref{secLit}, the maximality of PW($d$) follows from \eqref{eq:ineqPW} which is proved via coupling arguments.
Theorem \ref{thMaximal} provides an independent proof of this result.

\begin{corollary}
\label{cor:JSQPowerofd}
JSQ, uniform splitting, and PW($d$), $d \ge 2$, are maximal allocation policies.
\end{corollary}

\begin{proof}
Recall (\ref{eq:pU}), (\ref{eq:pJSQ}) and (\ref{eq:pPWd}).
As $\bp^{(s)} \preceq_{\textsc{gsc}} \bp^{(1)}$ (and  $\bp^{(1)} \preceq_{\textsc{gsc}} \bp^{(1)}$ by definition),
both the JSQ and uniform splitting policies satisfy the assumptions of Theorem \ref{thMaximal}.

To prove the statement for PW$(d)$ policies, $d \in \llbracket 2,s-1 \rrbracket$, fix such $d$ and observe that, for any $k \le s-d+1$,
the quantity $\sum_{i=k}^s p^{(d)}_i$ is the probability that the
$d$ uniformly drawn servers have indices in $\llbracket k,s \rrbracket$, which is equal to ${s-k+1 \choose d} /{s \choose d}.$
From this, we deduce that
\begin{equation}
\label{eq:compared2}
\bp^{(d)} \preceq_{\textsc{gsc}} \bp^{(2)}.
\end{equation}
Indeed, for any $k\ge s-d+2$ we have $\sum_{i=k}^s p^{(d)}_i =0$, whereas for any $k \le s-d+1$, we have that
\[{\sum_{i=k}^s p^{(d)}_i \over \sum_{i=k}^s p^{(2)}_i} = {{s-k+1 \choose d}{s \choose 2} \over {s \choose d}{s-k+1 \choose 2} } = {(s-d)...(s-d-k+2) \over (s-2)...(s-2-k+2)}  \le 1,\]
whence (\ref{eq:compared2}).
Now, $\sum_{i=s}^s p^{(2)}_i = 0$ and for all $k\le s-1$, so that
\[\sum_{i=k}^s p^{(2)}_i = {1 \over {s \choose 2}} \sum_{i=k}^{s} (s-i) = {s-k \over s-1}{s-k+1 \over s} \le {s-k+1 \over s} = \sum_{i=k}^s{1\over s},\]
implying that $\bp^{(2)} \preceq_{\textsc{gsc}} \bp^{(1)}$. This, together with (\ref{eq:compared2}) and the transitivity of ``$\preceq_{\textsc{gsc}}$", shows that
$\bp^{(d)} \preceq_{\textsc{gsc}} \bp^{(1)}$. Thus, PW($d$) is maximal by Theorem \ref{thMaximal}. 
\end{proof}

Theorem \ref{thJ2SQp} and Proposition \ref{prop:busyvsidling} also imply
\begin{corollary}
\label{cor:maximalnonidling}
$\maS^\ni(\bp) = [0,1)$ for any $\bp$ satisfying (\ref{Ncond}). 
In particular, the non-idling versions of uniform splitting and PW($d$) allocation policies are maximal.
\end{corollary}

\vspace{.5cm}
\section{Insufficiency of the Condition $\rho < 1$}
\label{secInsufficiency}
Theorem \ref{thMaximal} requires, in addition to the usual traffic condition $\rho < 1$, that the allocation probability $\bp$ is smaller, in the generalized Schur convex order,
than the uniform probability distribution on $\llbracket 1,s \rrbracket$.
We now demonstrate that the latter condition is not futile, and that the traffic condition by itself does not imply stability of a system.
To provide simple counter-examples, we consider $\bp_{p,2}$-allocation probabilities, with $\bp_{p,2} := \left(1-p,p,0,...0\right)$, for $0<p<1$. In other words, any arrival is
routed to the shortest queue with probability $q := 1-p$,
or to the second-shortest queue with probability $p$ (ties broken by a uniform draw from the relevant queues.)
We interpret $p$ as the probability that the controller (or the arriving customer) is making an error in distinguishing between the shortest and the second shortest queue.
We denote this $\bp_{p,2}$-allocation policy by J2SQ$(p)$, and its corresponding non-idling version by J2SQ$^\ni(p)$.

Under the non-idling version of the latter policy, the controller identifies idle servers, but otherwise has a probability $p$ of making an error by sending an arrival to the second-shortest queue.
Thus, when all the servers are busy, errors are made according to a Bernoulli trial with a probability $p$ of ``success.''
Observe that, for $p^{(1)}$ in \eqref{eq:pU},
\bequ \label{suff2}
\bp_{p,2} \preceq_{\textsc{gsc}} \bp^{(1)} \qiffq p \le 1-1/s.
\eeq

For a given number of servers $s \ge 1$ and an error probability $p > 0$, let 
\begin{equation}
\label{eq:defVcr}
\Vcr(p) := {s-1 \over 2 s} \left(1+\sqrt{1+{4 \over p(s-1)}}\right).
\end{equation}
We refer to $\Vcr(p)$ as the {\em critical value} (for stability; see Theorem \ref{thJ2SQp} below).
Simple algebra shows that
\bequ \label{Vcr-cond}
\Vcr(p) < 1 \qiffq p>1-1/s.
\eeq
\begin{theorem}\label{thJ2SQp}
$\maS^\ni(\bp_{p,2}) \subset [0, \Vcr(p)\wedge 1)$ for any $p\in [0,1]$. 
\end{theorem}
We defer the proof of Theorem \ref{thJ2SQp} to \S \ref{secProof-J2SQp}.
In view of (\ref{suff2}) and \eqref{Vcr-cond}, Theorems \ref{thMaximal} and \ref{thJ2SQp} immediately imply the following,
\begin{corollary} \label{coroJ2SQp}
J2SQ$^\ni(p)$ is maximal if and only if $p \le 1 - 1/s$.
\end{corollary}

In view of Proposition \ref{prop:busyvsidling}, Corollary \ref{coroJ2SQp} implies that the stability region under the $\bp_{p,2}$-allocation policy is also characterized by the value of $p$.
\begin{corollary} \label{coroJ2SQid}
$\maS(\bp_{p,2}) \subseteq [0, \Vcr(p)\wedge 1)$ for all $p \in [0, 1]$. In particular J2SQ$(p)$ is maximal if and only if $p \le 1 - 1/s$.
\end{corollary}

\subsection{Join the $2$nd Shortest Queue Allocation Policy}
\label{subsec:J2SQ}
The proof of Theorem \ref{thJ2SQp} involves some technical details that obscure the main intuition for the instability whenever the error probability $p$
is greater than $1-1/s$.
Simplicity is achieved by consider the special case $p=1$, which is tantamount to having the allocation vector be $\bp_{1,2} := (0,1,0,...,0)$.
In this case, the routing policy is simply {\em join the second shortest queue}, which we denote by J2SQ; we denote its non-idling version by J2SQ$^\ni$.
It follows from \eqref{Vcr-cond} that $\Vcr(1)$, defined in (\ref{eq:defVcr}) with $p=1$, satisfies $\Vcr(1) < 1$.

\begin{proposition} \label{prop:J2SQ}
$\maS^\ni(\bp_{1,2}) \subset [0, \Vcr (1))$. 
In particular, J2SQ$^\ni$ is non-maximal.
\end{proposition}
\begin{proof}
Let
\bequ \label{setS}
\maA := \{x \in \ZZ_+^s : x_1 \in \{0,1\}, ~ x_i \ge 2, ~ i \in \llbracket 2, s\rrbracket \}, 
\eeq
and note that,
whenever exactly one of the servers has no jobs waiting in queue, the process $Q$ takes values in the set $\maA$, that is, if $Q_i(t) \in \{0,1\}$ for exactly one $i \in \llbracket 1,s\rrbracket$, then $\ord{Q(t)} \in \maA$.

Let $\mathbf s := (0,2,\dots,2) \in \maA$, and for $k = 1,2,\dots$, let $V_k$ denote the event that the $k$th visit of $\ord{Q} := \{\ord{Q(t)} : t \ge 0\}$ to $\maA$ starting at $\mathbf s$ occurs,
where that $k$th visit begins at time $t_k \ge 0$ if $\ord{Q(t_k-)} \ne \mathbf s$ and $\ord{Q(t_k)} = \mathbf s$, and ends when $\ord{Q}$ exists the set $\maA$, namely, at a random time $t_k+T_k$
such that $\ord{Q(t_k+T_k-)} \in S$ and $\ord{Q(t_k+T_k)} \notin \maA$.
We will henceforth refer to such a visit to $\maA$ (which begin at $\mathbf s$) simply as a ``visit'', and to $T_k$ as the length of the $k$th visit.

We prove the result by making the contradictory assumption that $Q$ is positive recurrent, and thus ergodic. Under this ergodicity assumption, $P(V_k) = 1$ for all $k \ge 1$, and
the lengths of the visits $\{T_k : k \ge 1\}$ are IID, by virtue of the strong Markov property, with $P(0 < T_1 < \infty) = 1$ and $E[T_1] < \infty$.
Now, during the $k$th visit, namely, during the intervals $I_k := [t_k, t_k+T_k)$, the ordered queue process $\ord{Q}$ operates as follows:
Any arrival is routed to server $1$, if this server is idle. Otherwise, the arrival is routed to server $2$.
Hence, over each interval $I_k$, we can view server $1$ as a single-server loss system (to which we refer as the ``{front server}''),
with the overflow from this front server constituting the arrival process to a system with $s-1$ homogeneous servers operating under the JSQ routing policy (to which we refer as the ``back servers'').

If the first arrival during the $k$th visit finds the system in state $\mathbf s$, 
then that arrival is routed to server $1$ (which is idle).
Let $A_k$ denote this latter event:
with $a_k$ denoting the time of the first arrival after time $t_k$, $A_k := \{Q(a_k-) =\mathbf s\}$. By the strong Markov property, the events $A_1, A_2, \dots$ are
independent and have the same probability, and it clearly holds that $P(A_1) > 0$.

By Lemma \ref{lmStat} in Section \ref{appendix}, the first arrival to a single-server loss system puts this system in steady state. In particular, on $[a_1,t_1+T_1)$
the instantaneous probability that an arrival finds server $1$ busy, and is therefore ``overflowed'' to the back system, is $\lm/(\lm+\mu)$.
Thus, due to the PASTA (Poisson Arrivals See Time Average) property,
the ``arrival rate'' to the back servers during $[a_1,t_1+T_1)$ is $\af:=\lm^2/(\lm+\mu)$. It follows that the process $\ord{Q_{-1}} := \ord{(Q_2,...,Q_s)}$ coincides in distribution with the ordered
queue-length process of a JSQ system with $s-1$ servers and arrival rate $\af$.

Next, observe that $\Vcr(1) < 1$ by \eqref{Vcr-cond}, and that $\Vcr(1)$ is thus the only positive root of the polynomial
$x \mapsto s^2x^2 -(s-1)sx -(s-1).$
It then readily follows that, for any $\rho > 0$,
\begin{equation}
\label{eq:equivregions}
{(s\rho)^2 \over 1+s\rho} > (s-1) \quad\mbox{ if and only if }\quad \rho > \Vcr(1).
\end{equation}
Therefore, if $\rho=\lambda/s\mu > \Vcr(1)$, then $\alpha > (s-1)\mu$, and so the probability that
the process $\ord{Q_{-1}}$ will never reach a state in which the smallest of the $s-1$ queues is equal to $1$ is strictly positive, implying that $P(T_1 = \infty) > 0$.
If $\af = (s-1)\mu$ (so that $\rho = \Vcr(1)$), then $\ord{Q_{-1}}$ is null recurrent, and the expected time until a state with the smallest queue being $1$ is reached is infinite.
In either case, the expected length of a visit is infinite, namely, $E[I_1] = E[T_1] = \infty$, in contradiction to the assumed ergodicity of $Q$.
\end{proof}

The proof of Proposition \ref{prop:J2SQ} makes the reason for the instability of the system we consider apparent:
Eventually, the system must split into a front loss single-server system whose overflow process constitutes the arrival process to a back $(s-1)$ parallel-server system operating under the JSQ policy.
If the overflow process is larger than the service capacity of the ``back servers'', then the system as a whole is unstable, because the expected time for it to exit this split structure is infinite.
In particular, once the system splits, the expected time until $Q$ reaches states that are not in the set $\maA$ defined in \eqref{setS} is infinite.
In fact, the regenerative structure of $Q$ implies that, if the traffic intensity is {\em strictly larger} than the critical value, i.e., if $\rho > \Vcr(p,s)$,
then $P(T_k = \infty \text{ for some } k \ge 1) = 1$ and $\|Q(t)\| \ra \infty$ w.p.1 as $t\tinf$.

\begin{remark}
\rm
We note that the (in)stability of the back system is solely determined by
the arrival rate to that system and mean service time $\mu$, and is independent of any other distributional assumptions. In particular, it does not rely on the service time distribution.
Furthermore, the blocking probability of a loss system is insensitive to the service-time distribution, so that the overflow rate from the front server {\em at stationarity} is $\af = \lm^2/(\lm+\mu)$
regardless of the assumption that service times are exponentially distributed. Thus, a generalization of Proposition \ref{prop:J2SQ} can be proved for a system with general service time distributions
having a finite mean $\mu$, but further arguments are needed for the step in which PASTA is applied.
\end{remark}

\subsection{Join the $m$-Shortest Queue Allocation Policy}
\label{subsec:JmSQ}
The arguments in the proof of Proposition \ref{prop:J2SQ} can be easily extended to the case in which there are several ``front servers'' instead of just one such server,
a scenario which arises when the $p$-allocation policy follows the ``join the $m$th shortest queue" assignment rule, corresponding to the allocation vector $\bp_{1,m} = (0,...,0,\underbrace{1}_m,0,...,0)$.
Under this allocation policy, which we denote by J$m$SQ, an incoming customer is routed to the $m$th shortest queue ($2 \le m \le s$) with probability $1$.
The non-idling version of this policy is denoted by J$m$SQ$^\ni$.

For $m \in \llbracket 2, s\rrbracket,$ define
\begin{align}
 \label{gSet}
\mathscr G(m) & := \left\{\rho \in (0,1)\,:\,{s\rho\left(s\rho\right)^{m-1} /(m-1)! \over \sum_{i=0}^{m-1}\left(s\rho\right)^{i} / i!} < (s-m+1) \right\}; \\
\Vcr(1,m) & := \sup \, \mathscr G(m). \label{eq:Vcr1m}
\end{align}
Note that the set $\mathscr G(m)$ is not empty, since it contains all the positive numbers that are smaller than $(s-m+1)/s$. In particular, $\Vcr(1,m)$ is finite.
Further, the inequality in the definition of $\mathscr G(m)$ reduces to (\ref{eq:equivregions}) when $m=2$, so that $\Vcr(1,2)\equiv \Vcr(1)$, for $\Vcr(1)$ in (\ref{eq:defVcr}).

\begin{lemma} \label{lemma:regionm}
$\Vcr (1,m) < 1$ for all $m \in \llbracket 2, s\rrbracket$.
\end{lemma}
The proof of Lemma \ref{lemma:regionm} appears is the appendix.

Given Lemma \ref{lemma:regionm}, the following result generalizes Proposition \ref{prop:J2SQ}.
\begin{proposition}
\label{prop:JmSQ}
$\maS^\ni(\bp_{1,m}) \subset [0, \Vcr(1,m))$; In particular, J$m$SQ$^\ni$ is non-maximal.
\end{proposition}

\begin{proof}{Proof of Proposition \ref{prop:JmSQ}}.
Fix $m \in \llbracket 2, s\rrbracket$ and let 
$$\maA_m := \{x \in \ZZ_+^s : x_i \in \{0,1\}, ~ i \in \llbracket 1, m-1\rrbracket, ~ \qandq x_j \ge 2, ~ j \in \llbracket m, s\rrbracket \}.$$
As in the proof of Proposition \ref{prop:J2SQ}, the statistical homogeneity of the $s$ servers implies that any vector $\bx \in \ZZ^s_+$ that has exactly $m-1$ coordinates with values in $\{0,1\}$
can be considered in $\maA_m$ since $\ord{\bx} \in \maA_m$. Further, as long as the system is in $\maA_m$, it is essentially split into two systems: the first $m-1$ servers
operate like an $M/M/(m-1)$ loss system, and the remaining $s-m+1$ servers operate like a parallel system under the JSQ routing policy, whose arrival process
is the overflow from the first $m-1$ ``front servers.''
Let $\bs = \left(\underbrace{0,\dots,0}_{m-1},\underbrace{2,\dots,2}_{s-m+1}\right).$
We say that a {\em visit} begins when the system transitions into state $\bs$, and ends when it exists the set $\maA_m$, namely, when the splitting into a front and back servers ends.

Let $L_m := \{L_m(t) : t \ge 0\}$ denote the number-in-system process in the $M/M/(m-1)$ loss system, and let $L_m(\infty)$ denote a random variable having the stationary distribution of $L$, which we denote by $\pi_m$,
i.e., $\pi_m(j) := P(L_m(\infty) = j)$.
Note that, during a visit, the number of busy servers in the aforementioned $m-1$ front-servers is distributed like $L_m$. By Lemma \ref{lmStat-m} in \S \ref{appendix},
there exists a random time $\tau$, such that $L_m(t) = L_m(\infty)$ for all $t \ge \tau$, and therefore, the number of busy servers among those front servers is also distributed like $L_m(\infty)$
for all $t\ge \tau_k$ on the event $E_k := \{\tau_k < T_k\}$, where $T_k$ denotes the length of the $k$th visit, and $\{\tau_k : k \ge 1\}$ are IID with $\tau_1 \deq \tau$.
By the strong Markov property, all the visits are IID and $P(E_1) > 0$. Therefore, $\{E_k : k \ge 1\}$ must occur infinitely often, unless one of the visits is infinite,
i.e., finitely-many $E_k$'s will occur if and only if $T_k = \infty$, for some $k \ge 1$.

Now, if $E_k$ occurs for the $k$th visit, then the overflow process from the front servers, which is the arrival process into the back servers, has rate
$\lm \pi_m(m-1)$ after time $\tau_k$, due to PASTA. If $\rho \ge \Vcr(1,m)$, then $\lm \pi^m(m-1) \ge \mu(s-m+1)$, i.e. the arrival rate to the ``back servers'' is
larger than the maximum total service rate of those $s-m+1$ servers after time $\tau_k$ as long as the $k$th visit is in process.
Therefore, $P(T_k = \infty) > 0$ on the event $E_k$.
We conclude that
$$P(T_k = \infty \text{ for some } k \ge 1) = 1,$$
so that $Q$ is either transient or null recurrent.
\end{proof}

\subsection{Proof of Theorem \ref{thJ2SQp}} \label{secProof-J2SQp}
The proofs of Propositions \ref{prop:J2SQ} and \ref{prop:JmSQ} build on the fact that each time a splitting of the system occurs, the front ``loss system'' has a positive probability of reaching stationarity in finite time,
after which PASTA is employed to characterize the overflow rate into the ``back servers.''
In the setting of Theorem \ref{thJ2SQp} with $p<1$ the splitting is as follows: There is one ``front server'' and $s-1$ ``back servers'', as in the proof of Proposition \ref{prop:J2SQ}.
However, the front server does not operate as a loss system. Instead, during each ``visit'' (splitting event), the front server operates as an $M/M/1$ queue with an infinite buffer, having a Poisson arrival process with rate $\lm$.
Each arrival to this $M/M/1$ queue enters service if the server is idle, and otherwise joins its queue with probability $p$, and the back servers with probability $1-p$, independently of everything else.
In particular, the arrival process to the $s-1$ back servers constitutes all the arrival who did not join the front server.
For the particular $M/M/1$ queue we obtain during a splitting event, the time to reach stationarity is infinite, so that PASTA cannot be directly employed as in the proofs of Propositions \ref{prop:J2SQ}
and \ref{prop:JmSQ}.

\begin{proof}[Proof of Theorem \ref{thJ2SQp}.]
Consider $p \in (1-1/s,1]$, and fix $\lambda,\mu$ such that $\rho=\lambda/s\mu \in [\Vcr(p,s), 1)$.
Let $Y^\sf(t)\in\Z_+$ be the number of customers in the front server at time $t$, and for $i \in \llbracket 1,s-1 \rrbracket$, let $Y^{\ni}_i(t)$ be the size of the
$i$th queue among the back servers, in the increasing order of queue lengths. It is easily seen that both processes $Y^{\sf}$ and $Y:=\left(Y^{\sf},Y^{\sb}_1,...,Y^{\sb}_{s-1}\right)$ (as functions of $t$) are CTMCs
on $\Z_+$ and $\Z_+^{s-1}$, respectively.
In particular, $Y^{\sf}$ is a Birth and Death (BD) process on $\ZZ_+$ with respective birth and death rates
$\lm$ and $0$ at state $0$, and $\lm(1-p)$ and $\mu$ at all other states. By the assumed values of $p$ and $\rho$, $Y^{\sf}$ is ergodic with stationary distribution
\begin{align*}
\pi^{\sf}(0) &={\mu-\lambda+\lambda p \over \mu +\lambda p};\\
\pi^{\sf}(i) &=\left({\lambda (1-p)\over \mu}\right)^{i-1} {\lambda \over \mu}\pi^{\sf}(0),\,i \ge 2.
\end{align*}
In particular the stationary probability that the front server is busy is
\begin{equation}
\label{eq:defpip}
\pi^{\sf}\left(\ZZ_+^*\right)=1-\pi^{\sf}(0)= {\lambda\over \mu +\lambda p}={s\rho\over 1 + s\rho p}.
\end{equation}
Now, it is well-known that an ergodic BD process with birth and death rates that are uniformly bounded is exponentially ergodic; e.g., see \cite[\S 4]{Tweedie81}.
Then letting $\|\cdot\|_{TV}$ denote the total-variation norm (e.g., see \cite{asmussen}),
\bequ\label{ergodic}
\|P(Y^{\sf}(t) \in \cdot) - \pi(\cdot)\|_{TV} < C_0 e^{-\beta t}, \quad t \ge 0,
\eeq
for some $C_0 \in [0, \infty)$ that depends on the initial condition only, and for some $\beta > 0$ that is independent of the initial condition.

For a given $y \in \ZZ_+$, Let $P^y_t$ denote the one-dimensional marginal distribution of the random variable $Y^{\sf}(t)$ when $Y^{\sf}(0) = y$.
It follows from \eqref{ergodic} that, for any $\ep > 0$, there exists a $T^y_\ep < \infty$ that depends on the initial condition $y$,
such that
\bequ \label{ergodic2}
\|P^y_t - \pi^{\sf}\|_{TV} < \ep \qforallq t > T_\ep.
\eeq
Next, let $\{t_n : n \ge 1\}$ denote the event (arrival) times in the Poisson arrival process to the system, and for $A \subset \ZZ_+$, let

\bes 
P^{\sf}_n(A) := P(Y^{\sf}(t_n-) \in A) = P_{t_n-}(A) ~~\qandq~~ P^{\sf}_\infty(A) := \lim_{n\tinf} P^{\sf}_n(A) .
\ees
From the PASTA property, we know that the above limit $P^{\sf}_\infty$ exists for all $A \subset \ZZ_+$, and that $P^{\sf}_\infty = \pi$.
Thus \eqref{ergodic2} implies that, for any $\ep > 0$ and for any fixed initial condition $y$,
there exists $T^y_\ep$, such that $\|P^{\sf}_n - \pi^{\sf}\|_{TV} < \ep$ for all $n$ such that $t_n \ge T_\ep$.
(The weak convergence to the stationary distribution is equivalent to convergence in total variation since the state space of $T^{\sf}$ is countable.)
In particular, taking $A := \ZZ_+^*$--corresponding to the event that the front server is busy--and $Y^{\sf}(0) = 0$, we have that, for some $T_\ep := T^0_\ep > 0$
\bequ \label{almostPasta}
\left|P^{\sf}_n\left(\ZZ_+^*\right) - \pi^{\sf}\left(\ZZ_+^*\right)\right| < \ep \qforallq n \text{ for which } t_n > T_\ep.
\eeq

Let $N_{\textsc{of}}(a,b]$ denote the overflow process from the front server (which is the arrival process to the back servers) over the time interval $(a,b]$, $0 \le a < b$.
Consider also a sequence of independent Bernoulli random variables $\{B_n : n \ge 1\}$, that are also independent of all other random variables defining the system,
each having ``success'' probability $p$, i.e., $P(B_n = 1) = p$ for all $n \ge 1$.
As in (\ref{eq:equivregions}), one can easily check that $\rho >\Vcr(p,2)$ if and only if
$\lambda p\pi^{\sf}(\ZZ_+^*) > (s-1)\mu$. Take $\ep > 0$ for which $\lm p \left(\pi^{\sf}\left(\ZZ_+^*\right) - \ep\right) > (s-1)\mu$.
Then \eqref{almostPasta} implies that, for $T_\ep$ in \eqref{almostPasta} and for all $t > 0$,
\bequ \label{Pasta}
t^{-1}E\left[N_{\textsc{of}}(T_\ep,T_\ep + t]\right] = t^{-1} E \left[ \sum_{t_n \in (T_\ep, T_\ep+t]}\mathbf 1_{\left\{\{Y^{\sf}(t_n-) \in \ZZ_+^*\} \cap \{B_n = 1\}\right\}} \right]
> \lm p \left(\pi^{\sf}\left(\ZZ_+^*\right) - \ep\right).
\eeq
The rest of the proof is similar to the arguments in the proof of Proposition \ref{prop:J2SQ}:
Taking the (contradictory) assumption that $Q$ is ergodic, a splitting to a forward and backward servers must occur i.o.
Letting a visit begin when, during such a splitting, the front server first reaches the empty state, we have that
the visits are IID, and each lasts for at least $T_\ep$ time units with a strictly positive probability, for any $\ep$ satisfying the inequality in \eqref{Pasta}.
(Note that, since a visit begins at a fixed state, we can choose the same $T_\ep$ in \eqref{almostPasta} for all the visits.)
More specifically, with $I_k$ denoting the time interval during the $k$th visit beginning when the front server is empty and ending when the visit ends,
we have that $P(I_k > T_\ep) > 0$, so that $\{I_k > T_\ep\}$, $k \ge 1$, must occur i.o.
However, since the overflow process from the front server is guaranteed to be larger than the total service rate $\mu(s-1)$ of the back servers after time $T_\ep$,
there is a positive probability that a visit will never end, contradicting the ergodicity assumption. The proposition is proved. 
\end{proof}

\vspace{.5cm}
\section{Simulation Experiments for Workload-Based Allocation Policies} \label{sec:simu}
As discussed in Section \ref{secMotivation}, our results and analyses provide insights for systems operating under allocation policies that are based on the workload (as opposed to the queue length).
Indeed, it is intuitively clear from the proofs of our main results that a system under JSW also experiences random ``splitting'' into forward and backward subsystems,
and that the backward subsystem may be unstable (so that the whole system is unstable) even if $\rho < 1$.
In this section we present simulation experiments to support this intuition.
In fact, the simulations indicate that the bounds we obtained for the stability regions in Theorem \ref{thJ2SQp} and
Propositions \ref{prop:J2SQ} and \ref{prop:JmSQ}, are tight estimates of the stability regions for the corresponding workload-based allocation policies, which are formally defined as follows.

\begin{definition}
For $m \in \llbracket 1,s \rrbracket$ and $p \in [0,1]$, we say that the allocation policy is {\em Join the $m$th shortest workload with probability $p$}, denoted by
JmSW($p$), if each arrival is sent to the queue having the smallest workload with probability $1-p$, and is otherwise sent to the queue with the $m$th smallest workload with probability $p$.

In the non-idling version of JmSW($p$), denoted by JmSW$^{\ni}$($p$), an arrival is sent to an idle server w.p.1, if such a server is available, and is otherwise routed to a server according to JmSW($p$).
\end{definition}

\paragraph{Cases Considered.}
We simulated a system with $4$ servers, each providing exponentially distributed service with mean $1$,
that is operating under J2SW$^\ni$($p$) (join the second-smallest workload with probability $p$), where $p \in \{0.8, 0.9, 1\}$.
In addition, we simulated the system when it is operating under J3SW$^\ni$($1$), namely, $m=3$ and $p=1$.
For each of these four systems we simulated the corresponding embedded DTMC over $10^7$ arrivals for two values of the traffic intensity $\rho$, one that is slightly above, and the other slightly below,
the critical values $\Vcr(p)$ (for J2SW$^\ni$($p$)) and $\Vcr(1,3)$ (for the system under J3SW$^\ni$($1$)).
The critical values are computed via \eqref{eq:defVcr} and \eqref{gSet}-\eqref{eq:Vcr1m}, respectively.
In particular, for each of the four examples we considered a traffic intensity that is larger than the critical value of $\rho$ by $2/10^3 = 0.002$,
and a traffic intensity that is smaller than the corresponding critical value by $0.002$.
We emphasize that the critical values are for the same system operating under J2SQ$^\ni$($p$) and J3SW$^\ni$($1$), and so we do not know whether they are also the critical values for the system
under the simulated scenarios. 

In Figure \ref{Fig:plotJSW4} we show a sample path of the most loaded server (in terms of workload) for each of the six cases considered for the system
under J2SW$^\ni$($p$), namely, two examples, each with a different $\rho$ for each of the three different values of $p$, as described above.
Two sample paths simulated for the system operating under J3SW$^\ni$($1$), one for each value of $\rho$, are shown in Figure \ref{Fig:plotJSW4bis}.

 \begin{figure}
  \includegraphics[scale=0.4]{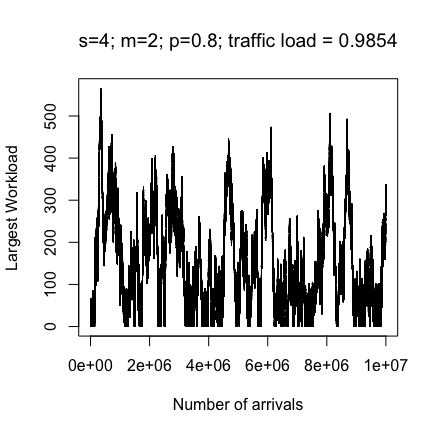}
  \hfill
  \includegraphics[scale=0.4]{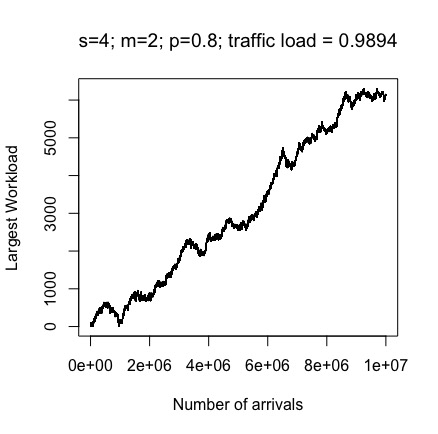}
   \vfill
  \includegraphics[scale=0.4]{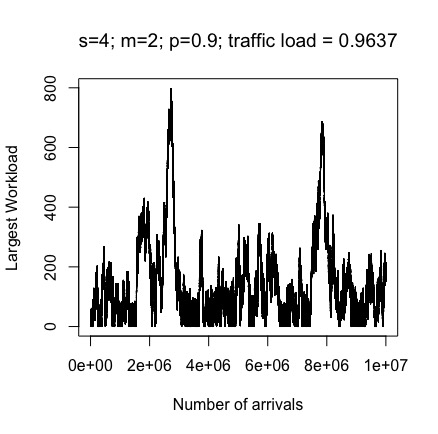}
  \hfill
  \includegraphics[scale=0.4]{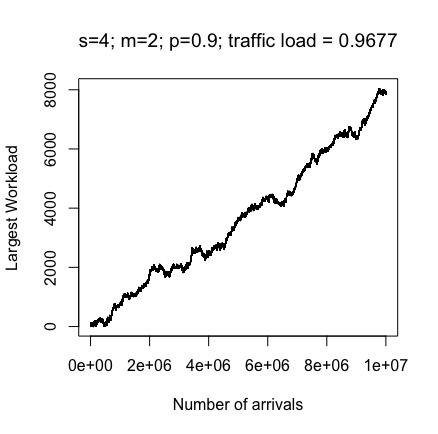}
  \hfill
   \includegraphics[scale=0.4]{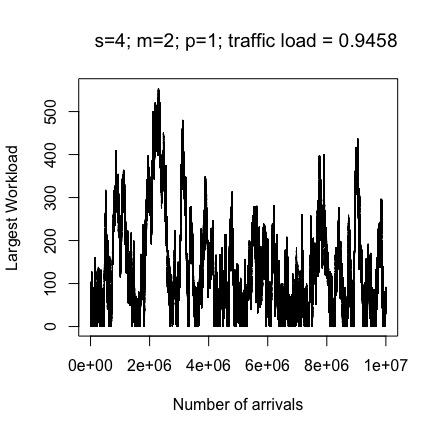}
  \hfill
  \includegraphics[scale=0.4]{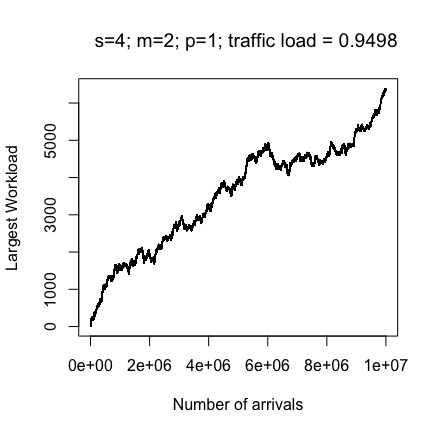}
  \vfill
 \caption{\scriptsize Sample paths of the largest workload process generated for $10^7$ arrivals of a system with four servers operating under J2SW$^\ni$($p$).
 The two figures in each row depict one value of $p$, with the left figure having $\rho = \Vcr(p) + 0.002$, and the right figure having $\rho = \Vcr(p) - 0.002$.
 {{\bf Upper panel:} a system operating under J2SW$^\ni$($0.8$), for which $\Vcr($0.8$) \approx 0.9874$.
 {{\bf Middle panel:} a system operating under J2SW$^\ni$($0.9$), for which $\Vcr(0.9) \approx 0.9657$.
 {\bf Lower panel:} a system operating under J2SW$^\ni$($1$), for which $\Vcr(1) \approx 0.9778$.}
 }
}
  \label{Fig:plotJSW4}
\end{figure}

\begin{figure}
  \includegraphics[scale=0.4]{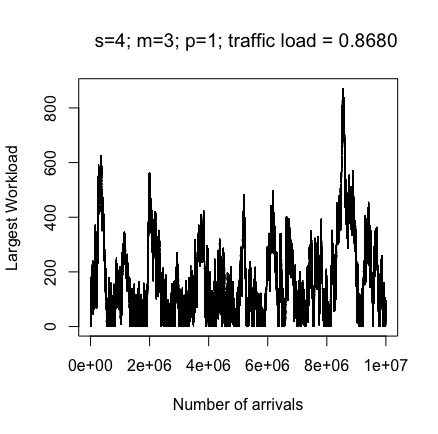}
  \hfill
  \includegraphics[scale=0.4]{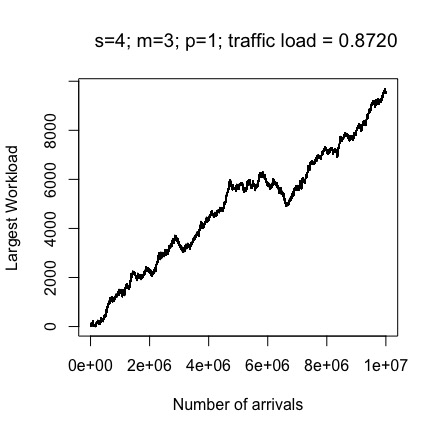}
 \caption{Sample paths of the largest workload process generated for $10^7$ arrivals of a system with four servers operating under J3SW$^\ni$($1$), for which $\Vcr(1,3) = 0.87$.
 The left figure depicts a sample path when $\rho = \Vcr(1,3) - 0.002$, and the right figure depicts a sample path when $\rho = \Vcr(1,3) + 0.02$.
 }
  \label{Fig:plotJSW4bis}
\end{figure}
We remark that, whenever $\rho$ is equal to its critical value, the queue process is null recurrent, and it is therefore hard to determine from simulation whether a system is stable or not when $\rho$
is ``too close'' to its critical value. (For any value of $\rho$ in a small-enough neighborhood of the critical value, the stochastic fluctuations are large, and one may observe a return to the empty state
over any finite time interval, even in the transient case.)
Nevertheless, for each of the four simulated routing policies, the system seems to be unambiguously unstable for the larger value of $\rho$, and to be stable for the smaller value of $\rho$.
This, together with the fact that the difference between the two traffic intensities is just $0.004$, strongly suggest that the critical value of $\rho$ for the system operating under the queue-based allocation policy
is very close, if not equal, to critical value of $\rho$ for the system operating under the corresponding workload-based allocation policy.

\vspace{.5cm}
\section{Summary}
In this paper we considered parallel server systems with $s \ge 1$ statistically homogeneous servers, to which jobs are routed upon arrival according to a family of random-assignment rules,
which we called $\bp$-allocation policies. That family of routing policies includes the PW($d$) routing rule, and the special cases JSQ and uniform routing, as well as their ``non-idling'' versions,
under which an arrival is always routed to an idle server, if there is one.
Our motivation for this study was the fact that in practice, and unlike the ideal settings that are typically considered in the literature, routing errors are likely to occur, so that jobs are not necessarily routed
to the shortest queue when JSQ is implemented.

We first characterized a sufficient condition for stability (in Theorem \ref{thMaximal}) which, in addition to the usual traffic condition $\rho < 1$, requires the $\bp$-allocation vector to be smaller,
in the generalized Schur convex order, defined in Definition \ref{defOrder}, than the uniform distribution on $\llbracket 1,s\rrbracket$. In particular, under the the extra assumption on $\bp$,
the $\bp$-allocation policy (and its non-idling version) is guaranteed to be maximal.

We then demonstrated that the condition $\rho < 1$ by itself does not guarantee that the system is stable, even when a non-idling $\bp$-allocation policy is employed.
Specifically, we considered the stability region of the policy J2SQ$^\ni(p)$, under which arrivals are always routed to an idle server, if one is present,
and are otherwise routed to the shortest queue with probability
$1-p$, and to the second shortest queue with an ``error probability'' $p$. Theorem \ref{thJ2SQp} proves that the stability region may be strictly contained in $[0,1)$,
namely, $\rho$ must be smaller than a positive number $\Vcr$, which is itself smaller than $1$ for a range of values of $p$.
Corollary \ref{coroJ2SQid} proves that $p$ must satisfy $p\le 1-1/s$ in order for J2SQ$^\ni(p)$ to be maximal.

One way of interpreting our results is that the risk of instability caused by erroneous routing decisions is small when the number of servers is large.
On the other hand, routing errors cause any system to effectively be in heavier traffic than planned;
if the system is designed to operate in ``heavy traffic,'' namely, if $\rho \approx 1$, then we can conclude that even a small probability of making routing errors may lead to harmful
departures from the desired performance, and may even lead to instability.

Finally, simulation examples in \S \ref{sec:simu} demonstrate that our results are insightful also for systems operating under JSW,
for which routing errors are more likely to occur, even in automated environments, because the actual workload in each queue can typically only be estimated.
Indeed, we conjecture that the stability regions under JSQ and JSW are the same.

\vspace{.5cm}

\newpage

\begin{appendix}{ }\label{appendix}

\section{Remaining Proofs}
\label{sec:proofs}

\subsection{Proof of Lemma \ref{lemma:order}}
\label{subsec:lemma}
As $\ba \preceq_{\textsc{gsc}} \bb$ and $\bx$ is ordered, for any $k \le s$ we have that
\begin{align*}
\sum_{i=k}^s x_ia_i &= x_ka_k + \sum_{i=k+1}^s \sum_{j=k}^{i-1} (x_{j+1}-x_j) a_i + \sum_{i=k+1}^s x_k a_i\\
                    &= x_k\sum_{i=k}^s a_i + \sum_{i=k+1}^s \left(x_i - x_{i-1}\right)\sum_{j=i}^s a_j
                    \le x_k\sum_{i=k}^s b_i + \sum_{i=k+1}^s \left(x_i - x_{i-1}\right)\sum_{j=i}^s b_j= \sum_{i=k}^s x_ib_i.
\end{align*}


\subsection{Proof of Lemma \ref{lemma:regionm}}
For $m \in \llbracket 2,s \rrbracket$ let $\pi_{\rho, m}$ denote the loss probability of a $M/M/m-1/0$ queue (a loss system with $m-1$ servers), having traffic intensity
$s \rho = \lm/\mu$; then
\[\pi_{\rho,m}:={\left(s\rho\right)^{m-1}/(m-1)! \over \sum_{i=0}^{m-1}\left(s\rho\right)^{i} /i!}.\]
Observe that $\rho \in \mathscr G(m)$, for $\mathscr G(m)$ in (\ref{eq:Vcr1m}), is equivalent to $s\rho\pi_{m}<(s-m+1)$. Also, we clearly have that
\begin{equation}
\label{eq:recregionm}
{1 \over \pi_{\rho,m+1}} = 1 + {m \over s\rho \pi_{\rho,m}},\,m = 2,...,s-1.
\end{equation}
First, $\Vcr(1,2) = \sup \mathscr G(2) < 1$ from (\ref{Vcr-cond}).
We then proceed by induction. Suppose that $\sup \mathscr G(m)<1$ for some $m \in \llbracket 2,s \rrbracket$.
Let $\rho \in \mathscr G(m+1)$. If $\rho \ge {(s-m)(s+1) \over (s-m+1)s}$, then we have that
\[s\rho \pi_{\rho,m+1}   < (s-m) \le s\rho{s-m+1 \over s+1}\]
which, after an immediate computation using (\ref{eq:recregionm}), is equivalent to $s\rho\pi_{\rho,m}<s-m+1$, i.e.
$\rho \in \mathscr G(m)$. By the induction assumption, this implies that
\[\sup \mathscr G(m+1) \le \left( \sup \mathscr G(m) \vee {(s-m)(s+1) \over (s-m+1)s}\right) <1,\]
which concludes the proof.

%
%
%

\section{Auxiliary results}
\label{sec:aux}

\label{subsec:coupling}
Let $L_1 := \{L(t) : t \ge 0\}$ denote the queue process in an $M/M/1/0$ queue (one-server loss system) having a Poisson arrival process with rate $\lm$ and service rate $\mu$.
The proof of the following lemma is a simple application of a standard coupling argument which we bring here for completeness.
\begin{lemma} \label{lmStat}
Consider the process $L_1$, and let $\tau_1$ denote the time of the first event after time $0$ (arrival or departure). Then $L_1$ is stationary for all $t \ge \tau_1$; in particular,
$P(L_1(t)) = 0) = 1 - P(L_1(t) = 0) = \mu/(\lm+\mu)$, $t \ge \tau_1$.
\end{lemma}

\begin{proof}
Let $L_e := \{L_e(t) : t \ge 0\}$ denote a stationary version of the process $L_1$, namely, $P(L_e(0) = 0) = 1-P(L_e(0) = 1) = \mu/(\lm+\mu)$. 
Let $T$ denote the first time $L_1$ and $L_e$ are equal; $T := \inf\{t \ge 0 : L(t) = L_e(t)\}$, and define the process
\bequ \label{Lx}
L_0(t) := \left\{\begin{array}{ll}
L_1(t) & t < T, \\
L_e(t) & t \ge T.
\end{array}\right.
\eeq
Since $T$ is a stopping time that is finite w.p.1, the strong Markov property implies that $L_0 \deq L_1$. The coupling inequality (e.g., \cite[VII 2a]{asmussen} gives
\bes
\|P(L_1(t) \in \cdot) - \pi(\cdot)\|_{TV} \le P(T > t).
\ees
Clearly, $L_0$ and $L_e$ are equal when the first event (arrival or departure) in either of the two processes occurs, and in particular, when the first event in $L_0$ occurs.
\end{proof}

Similarly to the proof of Lemma \ref{lmStat} we can prove the following result. Recall that $L_m :- \{L_m(t) : t \ge 0\}$ denotes the number-in-system process in an $M/M/(m-1)/0$ queue--a loss system
with $m-1$ servers and no buffer. Let $\tau_m := \inf\{t \ge 0 : L_m(t) = m-1\}$, namely, $\tau_m$ is the first time instant in which all servers are busy.
Note that $\tau_m$ is a proper random variable, i.e., $P(\tau_m < \infty) = 1$.
\begin{lemma} \label{lmStat-m}
If $L_m(0) = 0$, then $L_m$ is stationary for all $t \ge \tau_m$; in particular, for all $t \ge \tau_m$,
\bes
P(L(t) = k) = \pi^{m-1} := {\rho^k/k! \over \sum_{j=0}^{m-1} \rho^j/j!}, \quad k \in \llbracket 1, m-1 \rrbracket.
\ees
\end{lemma}

\begin{proof}
Let $L_\infty$ denote the stationary version of $L_m$, namely, $L_\infty(0) \deq \pi^m$, for $\pi^m$ in the statement of the lemma.
We couple $L_m$ and $L_\infty$ on the same probability space and allow them to evolve independently of each other until they couple, after which
the two processes follow the path of $L_\infty$ (similarly to the construction of $L_0$ in the proof of Lemma \ref{lmStat}). Since $L_m(0) = 0$, the two processes must have coupled by $\tau_m$,
and so the result follows from the strong Markov property.
\end{proof}

\end{appendix}

\providecommand{\bysame}{\leavevmode\hbox to3em{\hrulefill}\thinspace}
\providecommand{\MR}{\relax\ifhmode\unskip\space\fi MR }
\providecommand{\MRhref}[2]{%
  \href{http://www.ams.org/mathscinet^-getitem?mr=#1}{#2}
}
\providecommand{\href}[2]{#2}

\end{document}